\documentclass[11pt]{article}
\usepackage{amsthm}
\usepackage{fullpage}
\usepackage{amssymb, amsmath}
\usepackage{graphicx}
\usepackage{titlesec}
\usepackage{enumitem}
\usepackage{comment}
\usepackage{relsize}
\usepackage{comment}
\usepackage{xcolor}
\usepackage{mathtools}
\usepackage[font=small]{caption}
\usepackage{algorithm}
\usepackage{float}

\usepackage[backend=bibtex,giveninits=true,doi=false,url=false,isbn=false,style=trad-plain,hyperref,date=year]{biblatex}
\usepackage[hidelinks,colorlinks=true,linkcolor=blue,citecolor=blue]{hyperref}

\newtheorem{thm}{Theorem}

\newtheorem{lem}[thm]{Lemma}

\newtheorem{conj}[thm]{Conjecture}
\newtheorem{claim}{Claim}
\theoremstyle{definition}
\newtheorem{defin}[thm]{Definition}
\newtheorem{prob}[thm]{Problem}

\newtheorem{remark}[thm]{Remark}
\numberwithin{thm}{section}

\newcommand{\cU}{\mathcal{U}}
\newcommand{\cV}{\mathcal{V}}
\newcommand{\cW}{\mathcal{W}}

\newcommand{\recongraph}{\mathbf{RG}}

\newcommand{\colorfulcomplex}{\mathbf{Col}}
\newcommand*{\matchingconfig}{IMC}

\title{Tight constructions for reconfigurations of independent transversals}
\author{Ronen Wdowinski\thanks{Institute of Discrete Mathematics, Graz University of Technology, Steyergasse 30, 8010 Graz, Austria. Email: wdowinski@math.tugraz.at.}}
\date{\today}

\begin{document}

\maketitle

\begin{abstract}
For a graph $G$ and partition $\mathcal{U}$ of its vertex set, an independent transversal of $(G, \mathcal{U})$ is an independent set of $G$ that contains one vertex from each block of $\mathcal{U}$. Buys, Kang, and Ozeki studied when a reconfiguration graph on independent transversals of $(G,\mathcal{U})$ is connected, meaning any independent transversal can be transformed into any other one through a sequence of one-vertex modifications while always maintaining an independent transversal. Analogous to a theorem of Haxell, they proved that this is the case if $G$ has maximum degree $\Delta$ and each block of $\mathcal{U}$ has size at least $2\Delta$, except if the union of some $k \ge 1$ blocks of $\mathcal{U}$ induces $k$ disjoint copies of the complete bipartite graph $K_{\Delta, \Delta}$ in $G$. Solving one of their problems, we exactly characterize the partition structure in the latter exceptional instances of their theorem, showing that there is a rich variety of them but they are generated by a simple constructive procedure.
\end{abstract}

\section{Introduction}

Independent transversals in graphs are objects of great interest in many parts of combinatorics and discrete optimization, due to their applications to hypergraph matchings \cite{aharoni2000hall}, list colorings \cite{haxell2001note}, resource allocation \cite{asadpour2012santa, haxell2025improved}, and more. While previous work primarily focused on when independent transversals exist (e.g., \cite{aharoni2007independent, aharoni2015cooperative, haxell2011forming, haxell2019topological, haxell2024constructing, tang2025number}), a more recent line of inquiry by Buys, Kang, and Ozeki \cite{buys2025reconfiguration} studies when any two independent transversals can be transformed into one another through a sequence of single-vertex swaps within the space of independent transversals, i.e., when an associated reconfiguration graph is connected. These authors proved a sufficient maximum degree condition for the reconfiguration graph to be connected (Theorem \ref{thm:BKO} below) that is analogous to Haxell's theorem \cite{haxell1995condition, haxell2001note} for the existence of independent transversals. This puts independent transversals into the field of \textit{combinatorial reconfigurations}, a topic that has received significant attention and generally studies the structure and connectivity of discrete solution spaces, with applications to local search algorithms and to approximate sampling and counting schemes (see, e.g., \cite{mynhardt2019reconfiguration, nishimura2018introduction, van2013complexity} for surveys). In \cite{wdowinski2025hall}, this paper's author showed how topological methods, previously used to study the existence of independent transversals, can also be used to study their reconfigurations. 

In this paper, we solve a problem of Buys, Kang, and Ozeki \cite{buys2025reconfiguration} (Problem \ref{prob:main-problem} below) by giving an exact characterization of vertex-partitioned graphs that are extremal for their theorem, i.e., that show the tightness of their sufficient condition for a reconfiguration graph being connected (see Theorem \ref{thm:construction-universal}). Our proof method is combinatorial, building upon previous work of Haxell and the author \cite{haxell2024constructing} about vertex-partitioned graphs with no independent transversals. Comparing our work to \cite{haxell2024constructing}, this demonstrates further close ties between the topics of existence and reconfigurations of independent transversals which were exhibited in other ways in \cite{buys2025reconfiguration, wdowinski2025hall}.

Going into definitions, let $G$ be a graph and let $\cU = \{U_1, \ldots, U_m\}$ be a partition of $V(G)$. An \textit{independent transversal} (IT) of $(G,\cU)$ is an independent set $\{u_1,\ldots,u_m\}$ of $G$ that contains exactly one vertex $u_i$ from each block $U_i \in \cU$. The \textit{reconfiguration graph} $\recongraph(G,\cU)$ is the graph whose vertices are the IT's of $(G,\cU)$, with two IT's $S,T$ joined by an edge if $S \cup T$ is an independent set of size $|\cU|+1$. In other words, adjacent IT's $S, T$ in $\recongraph(G,\cU)$ agree on all but one block $U_i \in \cU$, and their vertices in $U_i$ are non-adjacent in $G$. Buys, Kang, and Ozeki \cite{buys2025reconfiguration} defined this rather restricted notion of adjacency in the reconfiguration graph for convenience in their inductive proof, but as indicated in \cite{wdowinski2025hall}, this definition is actually quite natural from a topological perspective: Viewing IT's as colorful simplices in the independence complex of $G$, they are adjacent in the reconfiguration graph precisely when they are common faces of a simplex with one more vertex.

A well-known theorem of Haxell \cite{haxell1995condition, haxell2001note} states that if $G$ has maximum degree $\Delta$ and $|U_i| \ge 2\Delta$ for all $i$, then $(G,\cU)$ always has at least one IT. Buys, Kang, and Ozeki \cite{buys2025reconfiguration} proved the following extension of Haxell's theorem to the reconfiguration setting.

\begin{thm}[Buys, Kang, Ozeki \cite{buys2025reconfiguration}] \label{thm:BKO}
Let $G$ be a graph with maximum degree $\Delta$, and let $\cU = \{U_1, \ldots, U_m\}$ be a partition of $V(G)$ such that $|U_i| \ge 2\Delta$ for all $i$. If $G[\bigcup_{i \in I} U_i]$ is not a disjoint union of $|I|$ copies of $K_{\Delta,\Delta}$, for all nonempty $I \subseteq [m]$, then $\recongraph(G, \cU)$ is connected.
\end{thm}

In particular, Theorem \ref{thm:BKO} implies that $\recongraph(G, \cU)$ is connected whenever $|U_i| \ge 2\Delta+1$ for all $i$. The proof of Theorem \ref{thm:BKO} in \cite{buys2025reconfiguration} was combinatorial, adapting ideas from \cite{graf2020finding} (related ideas also found in \cite{aharoni2007independent, haxell2006odd, haxell1995condition}), while the proof in \cite{wdowinski2025hall} was done using topological methods analogous to topological proofs of Haxell's theorem (see \cite{aharoni2007independent, aharoni2002triangulated, aharoni2015cooperative, haxell2019topological, meshulam2003domination}).

When $|U_i| = 2\Delta$ for all $i$, Theorem \ref{thm:BKO} gives only a partial characterization of when the reconfiguration graph $\recongraph(G, \cU)$ is (dis)connected: If $\recongraph(G, \cU)$ is disconnected, then Theorem \ref{thm:BKO} states that $G$ must contain a disjoint union of copies of the complete bipartite graph $K_{\Delta,\Delta}$, but it does not pin down precisely which partitions of disjoint unions of $K_{\Delta,\Delta}$ have disconnected reconfiguration graphs. Buys, Kang, and Ozeki \cite{buys2025reconfiguration} left it as an open problem to exactly characterize the minimal instances $(G,\cU)$ of Theorem \ref{thm:BKO} whose reconfiguration graphs are disconnected.

\begin{prob}[Buys, Kang, Ozeki \cite{buys2025reconfiguration}] \label{prob:main-problem}
Give an exact characterization of all instances $(G, \cU)$ in which $G$ is a disjoint union of $m$ copies of $K_{\Delta,\Delta}$, $\cU = \{U_1, \ldots, U_m\}$ is a partition of $V(G)$ with $|U_i| = 2\Delta$ for all $i$, $\recongraph(G, \cU)$ is disconnected, and $\recongraph(G - U_i, \cU - \{U_i\})$ is connected for all $U_i \in \cU$.
\end{prob}

Buys, Kang, and Ozeki \cite{buys2025reconfiguration} phrased Problem \ref{prob:main-problem} slightly differently, where they have the \textit{block-criticality} condition (i.e., that $\recongraph(G - U_i, \cU - \{U_i\})$ is connected for all $U_i \in \cU$) replaced by the condition that $(G, \cU)$ is \textit{irreducible}, meaning that we cannot find two disjoint vertex-partitioned subgraphs $(G_1, \cU_1)$ and $(G_2, \cU_2)$ satisfying $G = G_1 \cup G_2$ and $\cU = \cU_1 \cup \cU_2$. However, it is easy to show using Theorem \ref{thm:BKO} that these two minimality conditions are equivalent when $G$ is a disjoint union of copies of $K_{\Delta,\Delta}$ and $|U_i| = 2\Delta$ for all $i$. 

Our goal is to fully solve Problem \ref{prob:main-problem}. We refer to instances $(G,\cU)$ described by Problem \ref{prob:main-problem} as \textit{bad instances}. The simplest example of a bad instance $(G, \cU)$ is when $G = K_{\Delta, \Delta}$ and $\cU = \{V(G)\}$, as already noted in \cite{buys2025reconfiguration}. (When $G$ and $\cU$ are empty, we consider $\recongraph(G, \cU)$ to be an isolated vertex corresponding to the empty transversal.) Another example is when $G$ is the disjoint union of two copies of $K_{\Delta,\Delta}$ and $\cU$ is a standard bipartition of $G$. More generally, we have the following class of instances $(G, \cU)$ whose reconfiguration graphs $\recongraph(G,\cU)$ are disconnected (see Lemma \ref{lem:minimal-non-reconfigurable}).

\begin{defin} \label{defin:elementary}
We define an \textit{elementary non-reconfigurable instance} $(G, \cU)$ as follows. The graph $G$ is a disjoint union of $m$ complete bipartite graphs $K_{A_1, B_1} \sqcup \cdots \sqcup K_{A_m,B_m}$. Next, the partition $\cU = \{U_1, \ldots, U_m\}$ of $V(G)$ has the property that we could associate each block $U_i \in \cU$ with a unique complete bipartite component $K_{A_i, B_i}$ of $G$ such that $A_i \subseteq U_i$. Finally, we require that $\recongraph(G - U_i, \cU - \{U_i\})$ is connected for all $U_i \in \cU$.
\end{defin}

Despite their variety, elementary non-reconfigurable instances $(G, \cU)$ on disjoint unions of $K_{\Delta,\Delta}$ are not the only examples of bad instances. Indeed, one of the contributions of this paper is to introduce the following construction lemma that allows one to generate larger non-reconfigurable instances out of smaller ones.

\begin{lem} \label{lem:construction-intro}
Let $G$ and $H$ be disjoint graphs, and let $\cU = \{U_1, \ldots, U_m\}$ and $\cV = \{V_1, \ldots, V_n\}$ be partitions of $V(G)$ and $V(H)$, respectively. Consider the graph $G \cup H$ and any vertex partition of the form $\cW = \{U_1, \ldots, U_{m-1}, V_1', \ldots, V_n' \}$, where $V_1', \ldots, V_n'$ are obtained by distributing each vertex of $U_m$ into one of $V_1, \ldots, V_n$ arbitrarily. If 
\begin{itemize}
	\item[(a)] $\recongraph(G, \cU)$ is disconnected, and
	\item[(b)] $(H, \cV)$ has no IT, but $(H - V_i, \cV - \{V_i\})$ has an IT for any $V_i \in \cV$,
\end{itemize}
then $\recongraph(G \cup H, \cW)$ is disconnected.
\end{lem}

\begin{figure}
	\begin{center}
		\includegraphics[scale=0.8]{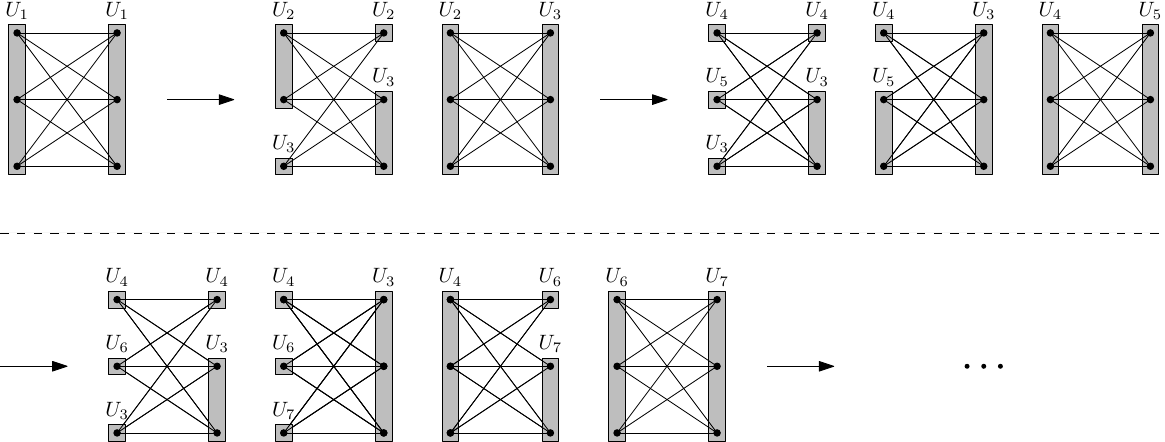}
	\end{center}
	\caption{Successive applications of Lemma \ref{lem:construction-intro}. At each step, we add a new copy of $(H, \cV) = (K_{A,B}, \{A,B\})$ with $|A|, |B| = 3$, choose a block $U_i$ from the previous graph, and distribute the vertices of $U_i$ into $A$ and $B$.}
	\label{fig:construction}
\end{figure}

See Figure \ref{fig:construction} for an example of successively applying Lemma \ref{lem:construction-intro}, where at each step $(H,\cV) = (K_{A,B}, \{A,B\})$ with $|A| = |B| = \Delta$. An elementary proof of Lemma \ref{lem:construction-intro} is given in Section \ref{subsec:construction-properties}. We show in Lemma \ref{lem:refined-construction} that under a mild additional condition, the instance $(G \cup H, \cW)$ in Lemma \ref{lem:construction-intro} retains the block-criticality property of instance $(G, \cU)$ (i.e., that deleting any block results in a connected reconfiguration graph). 

Our main result is that, in fact, all bad instances are obtained by starting with an elementary non-reconfigurable instance and then iteratively applying Lemma \ref{lem:construction-intro} with $(H,\cV) = (K_{A,B}, \{A,B\})$, similar to what is seen in Figure \ref{fig:construction}. This gives an exact characterization requested in Problem \ref{prob:main-problem}.

\begin{thm} \label{thm:construction-universal}
Every bad instance $(G, \cU)$ can be obtained using Lemma \ref{lem:construction-intro}. Specifically, $(G, \cU)$ is either an elementary non-reconfigurable instance of disjoint unions of $K_{\Delta,\Delta}$ with $|U_i| = 2\Delta$ for all $i$, or it can be obtained from such an elementary non-reconfigurable instance by iteratively applying Lemma \ref{lem:construction-intro} with $(H, \cV) = (K_{A,B}, \{A,B\})$, $|A| = |B| = \Delta$, and distributing $\Delta$ vertices of $U_m \in \cU$ into each of $A$ and $B$ at every iteration. 
\end{thm}

For a polynomial-time characterization of bad instances and how a given one can be reconstructed using Lemma \ref{lem:construction-intro}, see Remark \ref{rmk:poly-time}. We will actually prove a more general version of Theorem \ref{thm:construction-universal} (see Theorem \ref{thm:more-general}) in which the graph $G$ may consist of other complete bipartite graphs besides $K_{\Delta,\Delta}$ and there are no assumptions on the block sizes. We only insist that the number of components of $G$ and the number of blocks in $\cU$ are equal.

As indicated above, the work in this paper is analogous to but separate from previous work by Haxell and the author \cite{haxell2024constructing}. In the latter work, these authors introduced the analogue of Lemma \ref{lem:construction-intro} where if $(G, \cU)$ and $(H, \cV)$ both have no IT, then $(G \cup H, \cW)$ also has no IT \cite[Lemma 3]{haxell2024constructing}. They proved, in particular, that all instances $(G, \cU)$ with no IT, where $G$ is the disjoint union of $2\Delta - 1$ copies of $K_{\Delta,\Delta}$ and $|U_i| = 2\Delta-1$ for all $i$, can be generated using their construction lemma. This unified previous extremal constructions by Jin \cite{jin1992complete}, Yuster \cite{yuster1997independent}, and Szab\'o and Tardos \cite{szabo2006extremal}. Other applications of their construction lemma are found in \cite{cambie2024precise, haxell2024degree, wdowinski2026bounded}. Our proof of Theorem \ref{thm:construction-universal} will follow a similar technical combinatorial approach to \cite{haxell2024constructing}, but the details will be different, and in particular we will utilize combinatorial ingredients from \cite{buys2025reconfiguration, wdowinski2025hall} to deal with the reconfiguration setting. Our proof will be self-contained. Although the scope of our characterization in Theorems \ref{thm:construction-universal} and \ref{thm:more-general} is not remarkably wide, the careful combinatorial techniques used in this paper are of interest in further studies of reconfigurations of independent transversals and other related objects.

This paper is organized as follows. In Section \ref{sec:construction-lemma}, we describe a refined version of Lemma \ref{lem:construction-intro} (Lemma \ref{lem:refined-construction}) and generalized versions of Theorem \ref{thm:construction-universal} that we will prove (Theorems \ref{thm:more-general} and \ref{thm:equivalent-general}). Some details are deferred to Appendix \ref{sec:appendix}. In Section \ref{sec:main-proof}, we give our detailed combinatorial proof of Theorem \ref{thm:equivalent-general}, from which Theorems \ref{thm:construction-universal} and \ref{thm:more-general} follow. In Section \ref{sec:discussion}, we conclude with some remarks on other combinatorial applications of Lemma \ref{lem:construction-intro} as well as its topological aspects.

\section{Properties of the construction lemma and a generalized setting} \label{sec:construction-lemma}

In this preliminary section, we gather the relevant properties of our main construction lemma (Lemma \ref{lem:construction-intro}) for the proof of Theorem \ref{thm:construction-universal}. After that, we state our generalized version of Theorem \ref{thm:construction-universal} (Theorem \ref{thm:more-general}) as well as an equivalent formulation that we will actually prove (Theorem \ref{thm:equivalent-general}).

\subsection{Properties of construction lemma} \label{subsec:construction-properties}

We start this subsection with a proof of Lemma \ref{lem:construction-intro}, which was our main construction lemma.

\begin{proof}[Proof of Lemma \ref{lem:construction-intro}]
We show, assuming condition (b) and that $\recongraph(G \cup H, \cW)$ is connected, that $\recongraph(G, \cU)$ is also connected. Consider any two IT's $S, T$ of $(G, \cU)$. Say that $S \cap U_m$ lies in $V_j'$ and that $T \cap U_m$ lies in $V_k'$, where $1 \le j_0, j_1 \le n$. Applying condition (b), let $S'$ be an IT of $(H - V_j, \cV - \{V_j\})$, and let $T'$ be an IT of $(H - V_k, \cV - \{V_k\})$, so that $S \cup S'$ and $T \cup T'$ are IT's of $(G \cup H, \cW)$. Since $\recongraph(G \cup H, \cW)$ is connected, there is a reconfiguration sequence of IT's of $(G \cup H, \cW)$ from $S \cup S'$ to $T \cup T'$, say $S \cup S' = S_0, S_1, \ldots, S_{N-1}, S_N = T \cup T'$. The sequence $S = S_0 \cap V(G), S_1 \cap V(G), \ldots, S_{N-1} \cap V(G), S_N \cap V(G) = T$ can be modified into a reconfiguration sequence of IT's of $(G, \cU)$ from $S$ to $T$, which would finish the proof. Specifically, since $(H, \cV)$ has no IT by condition (b), the set $S_i \cap V(G)$ contains an IT of $(G, \cU)$ for all $0 \le i \le N$. For each maximal interval $[i_0, i_1]$ for which $S_i \cap U_m$ stays the same for all $i \in [i_0, i_1]$, we choose a single vertex $v \in S_{i_0} \cap U_m$ and replace the set $S_i \cap V(G)$ by the set $(S_i - U_m) \cup \{v\}$ in the sequence for all $i \in [i_0, i_1]$. Finally, we collapse consecutive copies of the same IT in the resulting sequence into a single IT. This gives the desired reconfiguration sequence from $S$ to $T$.
\end{proof}

Next, we will need a refined version of Lemma \ref{lem:construction-intro} to prove Theorem \ref{thm:construction-universal}. For that purpose, we introduce the following definitions.

\begin{defin}
Let $G, H$ be graphs, and let $\cU = \{U_1, \ldots, U_m\}, \cV = \{V_1, \ldots, V_n\}$ be partitions of $V(G), V(H)$ respectively. 
\begin{itemize}
	\item[(1)] We say that $(G, \cU)$ is \textit{minimally RGD} (reconfiguration graph disconnected) if $\recongraph(G, \cU)$ is nonempty and disconnected, but $\recongraph(G - U_i, \cU - \{U_i\})$ is connected for all $U_i \in \cU$.
	\item[(2)] We say that $(H, \cV)$ is \textit{minimally NIT} (no independent transversal) if $(H, \cV)$ has no IT, but $\recongraph(H - V_i, \cV - \{V_i\})$ is nonempty and connected for all $V_i \in \cV$.
\end{itemize}
\end{defin}

We note that the only examples of minimally NIT instances $(H, \cV)$ that are actually relevant for our main results are complete bipartite graphs $G = K_{A,B}$ with their standard bipartitions $\cU = \{A,B\}$. 

The following refined lemma explains how Lemma \ref{lem:construction-intro} (and its reversal) preserves the property of being minimally RGD. While the statement is intuitive, the proof is a tedious verification of definitions and hence is deferred to Appendix A.

\begin{lem} \label{lem:refined-construction}
Let $G$ and $H$ be disjoint graphs, and let $\cU = \{U_1, \ldots, U_m\}$ and $\cV = \{V_1, \ldots, V_n\}$ be partitions of $V(G)$ and $V(H)$, respectively. Consider the graph $G \cup H$ and any vertex partition of the form $\cW = \{U_1, \ldots, U_{m-1}, V_1', \ldots, V_n' \}$ where $V_1', \ldots, V_n'$ are obtained by distributing each vertex of $U_m$ into one of $V_1, \ldots, V_n$ arbitrarily. 
Assume that $(H, \cV)$ is minimally NIT. Then $(G, \cU)$ is minimally RGD if and only if $(G \cup H, \cW)$ is minimally RGD.
\end{lem}

In addition, we state the following lemma which is technically needed for all of our main results. The proof is also deferred to Appendix \ref{sec:appendix}. 

\begin{lem} \label{lem:minimal-non-reconfigurable}
If $(G, \cU)$ is a minimal non-reconfigurable instance, as in Definition \ref{defin:elementary}, then $\recongraph(G, \cU)$ is disconnected.
\end{lem}

\subsection{A generalized setting} \label{subsec:generalized-setting}

Below is our more general theorem version of Theorem \ref{thm:construction-universal}, where we do not require the complete bipartite components to be $K_{\Delta,\Delta}$ and we do not make assumptions on block sizes.

\begin{thm} \label{thm:more-general}
Let $G$ be a graph and let $\cU = \{U_1, \ldots, U_m\}$ be a partition of $V(G)$. Suppose that
\begin{itemize}
	\item[(a)] $(G, \cU)$ is minimally RGD,
	\item[(b)] $G$ is a disjoint union of complete bipartite graphs, and
	\item[(c)] $G$ has exactly $|\cU|$ connected components.
\end{itemize}
Then $(G, \cU)$ can be constructed using Lemma \ref{lem:construction-intro}, by starting with an elementary non-reconfigurable instance and iteratively adding complete bipartite graphs with the standard bipartition.
\end{thm}

From Theorem \ref{thm:more-general}, it is easy deduce Theorem \ref{thm:construction-universal}. Specifically, any instance $(G, \cU)$ from the hypothesis of Theorem \ref{thm:construction-universal} easily satisfies the requirements for Theorem \ref{thm:more-general}. Now we just need to check that at every iteration of applying Lemma \ref{lem:construction-intro}, the current graph $(G, \cU)$ must already have all its blocks with size $2\Delta$, and moreover the chosen block $U_m \in \cU$ must have $\Delta$ of its vertices be distributed to $A$ and the other $\Delta$ vertices distributed to $B$, where again $(H, \cV) = (K_{A,B}, \{A,B\})$ with $|A| = |B| = \Delta$. If one of these does not hold, then the resulting vertex-partitioned graph $(G \cup H, \cW)$ will have at least one block of size less than $2\Delta$, and it is easy to show that this would remain the case for all further iterations of Lemma \ref{lem:construction-intro}. Hence, Theorem \ref{thm:construction-universal} follows.

To prove Theorem \ref{thm:more-general}, we will actually prove the following equivalent theorem with a seemingly simpler conclusion.

\begin{thm} \label{thm:equivalent-general}
Let $G$ be a graph and let $\cU = \{U_1, \ldots, U_m\}$ be a partition of $V(G)$. Suppose that
\begin{itemize}
	\item[(a)] $(G, \cU)$ is minimally RGD,
	\item[(b)] $G$ is a disjoint union of complete bipartite graphs, and
	\item[(c)] $G$ has exactly $|\cU|$ connected components.
\end{itemize}
Then for every $U_i \in \cU$, there exists a complete bipartite component $K_{A,B}$ of $G$ such that $A \subseteq U_i$.
\end{thm}

It is easy to see that Theorem \ref{thm:more-general} implies Theorem \ref{thm:equivalent-general}. Now we show the converse.

\begin{proof}[Proof of Theorem \ref{thm:more-general} from Theorem \ref{thm:equivalent-general}]
Let $(G, \cU)$ be an instance satisfying conditions (a), (b), (c) of Theorem \ref{thm:more-general}. By Theorem \ref{thm:equivalent-general}, for every $U_i \in \cU$ there exists a complete bipartite component $K_{A,B}$ of $G$ such that $A \subseteq U_i$. If one can associate a unique such component $K_{A_i,B_i}$ to every $U_i \in \cU$, then $(G, \cU)$ is an elementary non-reconfigurable instance and we are done. Otherwise, there exists a complete bipartite component $K_{A,B}$ of $G$ and two distinct blocks $U_i, U_j \in \cU$ such that $A \subseteq U_i$ and $B \subseteq U_j$. Now we form the graph $G' \coloneqq G - K_{A,B}$ and vertex partition $\cU' \coloneqq \cU - \{U_i, U_j\} \cup (\{(U_i \cup U_j) - (A \cup B)\})$. That is, we delete $K_{A,B}$ from $G$ and we combine the remaining elements of $U_i$ and $U_j$ into a single block. Notice that $(G, \cU)$ can be obtained by applying Lemma \ref{lem:construction-intro} to $(G', \cU')$ and $(K_{A,B}, \{A,B\})$. Moreover, $(G', \cU')$ is again minimally RGD by Lemma \ref{lem:refined-construction}, and $G$ is a disjoint union of $|\cU'|$ complete bipartite graphs. By induction, $(G', \cU')$ can be constructed using Lemma \ref{lem:construction-intro}, by starting with an elementary non-reconfigurable instance and iteratively adding complete bipartite graphs with the standard bipartition. We conclude that the same is true for $(G,\cU)$.
\end{proof}

Theorems \ref{thm:more-general} and \ref{thm:equivalent-general} are analogues of \cite[Theorem 19 and Lemma 28]{haxell2024constructing}. We end this section with a remark about how instances $(G, \cU)$ satisfying conditions (a), (b), (c) of Theorem \ref{thm:more-general} can be recognized in polynomial time. 

\begin{remark} \label{rmk:poly-time}
Algorithm \ref{alg:bad-instance} below describes a polynomial-time scheme for recognizing bad instances $(G, \cU)$ associated to Problem \ref{prob:main-problem} (assuming the truth of Theorem \ref{thm:equivalent-general}). Note that in Step 6, we use the fact that $\recongraph(G,\cU)$ is automatically disconnected by Lemma \ref{lem:minimal-non-reconfigurable}. If we only care about whether $\recongraph(G,\cU)$ is disconnected, and not about the block-criticality, we may simply return YES at Step 6. Also, by keeping track of the deleted copies of $K_{\Delta,\Delta}$ and blocks, we may deduce how to reconstruct the original bad instance $(G,\cU)$ through iterative applications of Lemma \ref{lem:construction-intro}. The strategy of Algorithm \ref{alg:bad-instance} works more generally for recognizing instances $(G,\cU)$ satisfying the conditions of Theorem \ref{thm:equivalent-general}, but Step 6 requires modification because the equivalence of irreducibility and block-criticality does not hold in this general setting.

\begin{algorithm}
\caption{Recognizing bad instances} \label{alg:bad-instance}
\begin{itemize}
	\item[1.] Verify that $G$ is a disjoint union of $|\cU|$ complete bipartite graphs $K_{\Delta,\Delta}$ and that $|U_i| = 2\Delta$ for all $i$. If not, then return NO.
	\item[2.] Check that for every block $U_i \in \cU$, there is some complete bipartite component $K_{A_i,B_i}$ of $G$ satisfying $A_i \subseteq U_i$. If not, then return NO.
	\item[3.] If one can find a unique such component $K_{A_i,B_i}$ for every block $U_i$ (which is simply a bipartite matching problem), then skip to Step 6.
	\item[4.] Otherwise, find a component $K_{A,B}$ and distinct blocks $U_i, U_j$ such that $A \subseteq U_i$ and $B \subseteq U_j$.
	\item[5.] Update $G \leftarrow G - K_{A,B}$ and $\cU \leftarrow \cU - \{U_i, U_j\} \cup \{U'\}$, where $U' = (U_i \cup U_j) - (A \cup B)$. Go back to Step 1.
	\item[6.] Now we check that $(G, \cU)$ is an elementary non-reconfigurable instance. It is enough to check that $(G,\cU)$ is irreducible, i.e., we cannot find nonempty sub-instances $(G_1,\cU_1)$ and $(G_2,\cU_2)$ with $G = G_1 \cup G_2$ and $\cU = \cU_1 \cup \cU_2$. This can be done through a breadth-first search exploration of blocks connected by edges of $G$. If $(G, \cU)$ is irreducible then return YES, and otherwise return NO.
\end{itemize}
\end{algorithm}

\end{remark}

\section{Proof of Theorem \ref{thm:equivalent-general}} \label{sec:main-proof}

In this section, we prove Theorem \ref{thm:equivalent-general}. As explained in Section \ref{subsec:generalized-setting}, doing this would imply Theorem \ref{thm:more-general} and hence Theorem \ref{thm:construction-universal}. 

As mentioned before, the basic strategy and structure of our proof is similar to that of \cite[Lemma 28]{haxell2024constructing}, which is the analogous result in the existence setting and which utilized previous combinatorial work about IT's in graphs \cite{aharoni2007independent, haxell2006odd, haxell1995condition}. Roughly speaking, we will use the condition that the reconfiguration graph is disconnected, together with the other conditions of Theorem \ref{thm:equivalent-general}, to build some spanning combinatorial structure in our graph. This structure will basically be a special matching. Then we will show how properties and manipulations of this matching eventually yield the conclusion of Theorem \ref{thm:equivalent-general}. 

There will be some differences in our current proof. We will not be able to switch vertices along alternating paths as directly as in \cite{haxell2024constructing}, so our later analysis will be partly centered around finding a ``free" vertex that may enable similar switches, and consequently our proofs will be slightly lengthier. Despite such differences, the high level ideas from \cite{haxell2024constructing} transcribe remarkably smoothly. Our modified technical setup is inspired from \cite{buys2025reconfiguration, wdowinski2025hall}, where these ideas were used to combinatorially prove (variants of) Theorem \ref{thm:BKO}. Since we need to analyze the associated combinatorial structures rather closely, our exposition will be more elaborate.

For notation, the symmetric difference of two sets $S, T$ is denoted by $S \triangle T \coloneqq (S - T) \cup (T - S)$. For terminology, in a graph $G$ we say that a vertex set $X \subseteq V(G)$ \textit{totally dominates} a vertex set $Y \subseteq V(G)$ if every vertex of $Y$ is adjacent to some vertex of $X$. We use the same terminology when $Y$ is a vertex. Total domination is slightly different from the usual notion of graph domination, where the latter requires that every vertex of $Y$ either lie in $X$ or be adjacent to some vertex of $X$.

\subsection{Constructing induced matching configurations}

Let $G$ be a graph, let $\cU = \{U_1, \ldots, U_m\}$ be a partition of $V(G)$, and assume that $(G,\cU)$ has at least one IT. For a vertex subset $X \subseteq V(G)$, we denote 
\begin{equation*}
	I(X) \coloneqq \{i \in [m] : X \cap U_i \neq \emptyset\}.
\end{equation*}
For a vertex subset $X \subseteq V(G)$, we define the \textit{block graph} $\mathcal{G}_X$ to be the graph obtained from $G[X]$ by identifying all vertices of $U_i \cap X$ into a single vertex, which we still call $U_i$, for all $i \in I(X)$, while preserving all the induced edges. Actually, the block graph is generally a multigraph, as vertex identifications might introduce loops or parallel edges.

We introduce the following somewhat technical definition of a feasible 3-tuple, but it represents a rather intuitive concept. See Figure \ref{fig:feasible} for an illustration.

\begin{defin}
Given the instance $(G, \cU)$, a 3-tuple $(R, S, T)$ is called \textit{feasible} if it satisfies the following four properties:
\begin{itemize}
	\item[(i)] $R \subseteq V(G)$, and $S, T$ are IT's in two distinct connected components $C_0, C_1$ of $\recongraph(G, \cU)$ such that $|I(S \cap T)|$ is maximum among IT's from $C_0, C_1$;
	\item[(ii)] $S \triangle T \subseteq R$ and $I(R \cap S) = I(R \cap T) = I(R)$;
	\item[(iii)] $G[R - (S \triangle T)]$ is a forest whose components are $|R - (S \cup T)|$ non-trivial stars with centers lying in $R - (S \cup T)$ and remaining vertices lying in $S \cap T$;
	\item[(iv)] $\mathcal{G}_{R - (S - T)}$ is a forest on vertex set $\{U_i : i \in I(R)\}$, whose components are naturally rooted at vertices $U_j$ for $j \in I(S \triangle T)$.
\end{itemize}
\end{defin}

\begin{figure}
	\begin{center}
		\includegraphics[scale=0.75]{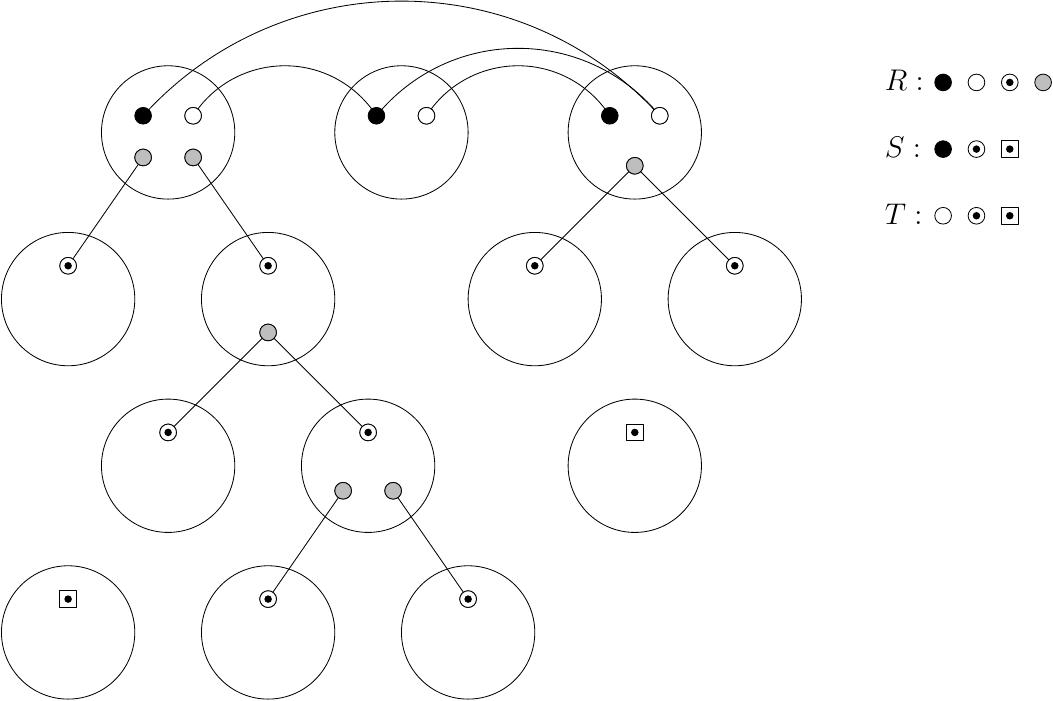}
	\end{center}
	\caption{A feasible 3-tuple $(R,S,T)$. The top three blocks represent the roots of the forest $\mathcal{G}_{R - (S - T)}$.}
	\label{fig:feasible}
\end{figure}

\begin{figure}
	\begin{center}
		\includegraphics[scale=0.75]{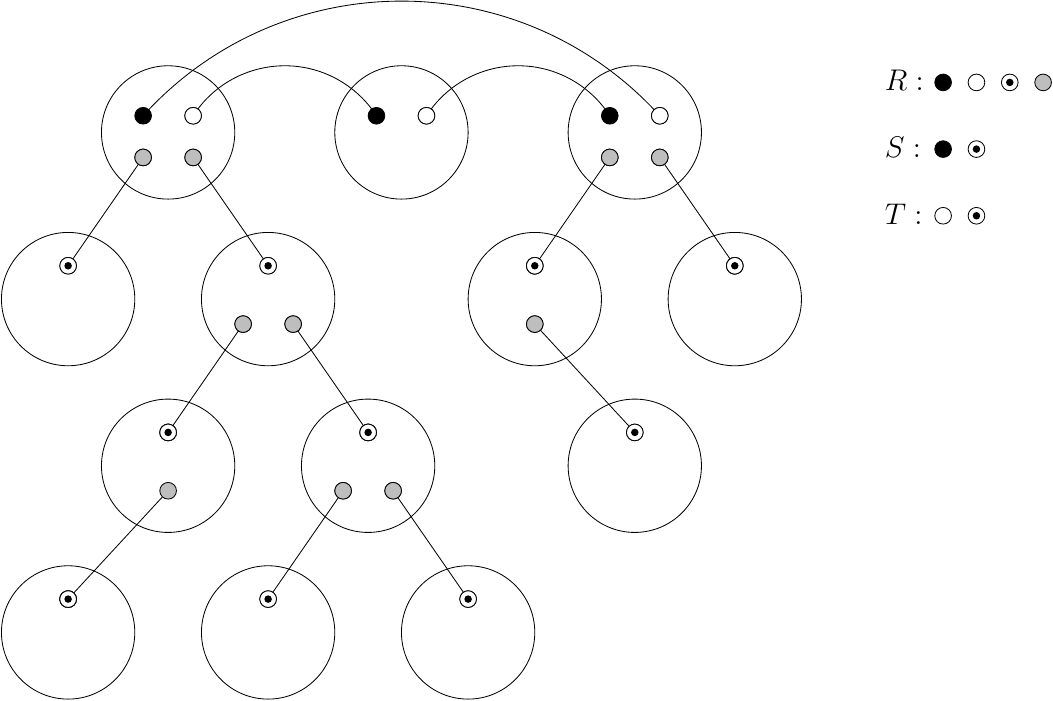}
	\end{center}
	\caption{An induced matching configuration $(R,S,T)$. We have $I(R) = I(S \cup T)$ and $G[R]$ is a matching.}
	\label{fig:matching-config}
\end{figure}

Assuming that the vertex-partitioned graph $(G, \cU)$ satisfies conditions (a), (b), (c) of Theorem \ref{thm:equivalent-general}, the goal in this subsection is to construct a special feasible 3-tuple $(R,S,T)$ that is defined below, analogous to but different from a concept in the existence setting \cite{haxell2006odd, haxell2024constructing, tang2025number} with the same name. See Figure \ref{fig:matching-config} for an illustration.

\begin{defin}
Given $(G, \cU)$, a feasible 3-tuple $(R, S, T)$ is called a \textit{induced matching configuration} (\matchingconfig) if it satisfies the additional properties that $I(R) = I(S \cup T) = [m]$ (so that $S \cup T \subseteq R$) and the induced subgraph $G[R]$ is a matching.
\end{defin}

Assume that $\recongraph(G,\cU)$ is disconnected. Fix any two distinct connected components $C_0, C_1$ of $\recongraph(G, \cU)$, and then choose any IT's $S \in C_0$, $T \in C_1$ that maximize $|I(S \cap T)|$. Then $(R,S,T) = (S \triangle T, S, T)$ is a feasible 3-tuple. Toward the goal of making it an \matchingconfig, we use Algorithm \ref{alg:growing-tuple} below to grow the vertex set $R$ in the feasible 3-tuple.

\begin{algorithm}
\caption{Growing feasible 3-tuples} \label{alg:growing-tuple}
\begin{enumerate}
	\item Start with the tuple $(R, S, T) \leftarrow (S \triangle T, S, T)$.
	\item While $R$ does not totally dominate $\bigcup_{i \in I(R)} U_i$:
	\begin{itemize}
		\item[a.] Choose any vertex $x \in \bigcup_{i \in I(R)} U_i$ not totally dominated by $R$.
		\item[b.] Among all IT's $Q_x$ of $(G[\bigcup_{i \in I(S \cap T)} U_i], \{U_i : i \in I(S \cap T)\})$ such that 
		\begin{itemize}
			\item[i.] $Q_x$ agrees with $S \cap T$ on the blocks indexed by $I(R)$,
			\item[ii.] there are no edges between $Q_x - R$ and $R$, and
			\item[iii.] $(S - T) \cup Q_x \in C_0$ and $(T - S) \cup Q_x \in C_1$,
		\end{itemize}
		choose one so that $x$ contains the minimum number of neighbors in $Q_x$.
		\item[c.] Update $R \leftarrow R \cup \{w\} \cup (Q_x \cap N(x))$ and $S \leftarrow (R \cap S) \cup Q_x$ and $T \leftarrow (R \cap T) \cup Q_x$ to form tuple $(R,S,T)$.
	\end{itemize}
	\item Return $(R,S,T)$.
\end{enumerate}
\end{algorithm}

Note that in Algorithm \ref{alg:growing-tuple}, although the IT's $S, T$ might change after a given iteration of the while loop, the index set $I(S \cap T)$ always remains the same. We prove in Claim \ref{claim-1} below that the set $Q_x$ in step 2b exists and satisfies $Q_x \cap N(x) \neq \emptyset$ at every iteration. Then, because the vertex set $R$ grows at every iteration, we get that Algorithm \ref{alg:growing-tuple} eventually terminates.

\begin{claim} \label{claim-1}
In Algorithm \ref{alg:growing-tuple}, a feasible 3-tuple is maintained at every iteration.
\end{claim}

\begin{proof}
Our initial 3-tuple $(S \triangle T, S, T)$ is feasible. Now assuming that $(R,S,T)$ is a feasible 3-tuple at the start of a given iteration, we show that the updated 3-tuple $(R',S',T') \coloneqq (R \cup \{x\} \cup (Q_x \cap N(x)), (S - T) \cup Q_x, (T - S) \cup Q_x)$ is also feasible, where $x$ is from step 3a and $Q_x$ is from step 3b. We observe that such an IT $Q_x$ exists for any chosen $x$ because the set $S \cap T$ satisfies the stated requirements. Properties (i) and (ii) of feasible 3-tuples are still satisfied by construction.

To show that property (iii) is satisfied, it is enough to show that $x$ has at least one neighbor in $Q_x$. If we can show this, then observing that $R' - (S' \triangle T') = (R - (S \triangle T)) \cup \{x\} \cup (Q_x \cap N(x))$ and recalling neither $x$ nor any vertex in $Q_x$ is not dominated by $R$, the graph $G[R' - (S' \triangle T')]$ will contain one more non-trivial star with center $x \in R' - (S' \cup T')$ and remaining vertices in $S' \cap T' = Q_x$. So suppose for contradiction that $x$ has no neighbors in $Q_x$, and that this is the earliest iteration for which this occurs. If $x \in U_j$ for some $j \in I(S \triangle T)$, then we could have chosen $S_0 \coloneqq S \cup \{x\} - \{s\}$ and $T_0 \coloneqq T \cup \{x\} - \{t\}$ at step 1 of the procedure, where $s \in U_j \cap S$ and $t \in U_j \cap T$, as these IT's satisfy $S_0 \in C_0$, $T_0 \in C_1$, and $|I(S_0 \cap T_0)| > |I(S \cap T)|$. This contradicts the original choices of $S$ and $T$. Thus, $x \in U_j$ for some $j \in I(R - (S \triangle T))$. Say that $v \in U_j \cap S = U_j \cap T$, and that $v$ has neighbor $y \in R$. This means that $y$ was chosen at an earlier iteration of step 3a, and $v$ lies in the IT $Q_y$ from step 3b at that iteration. But then we could have replaced $Q_y$ by $Q_x \cup \{x\} - \{v\}$, and $y$ has fewer neighbors in $Q_x \cup \{x\} - \{v\}$ than it does in $Q_y$, contradicting the choice of $Q_y$ at that earlier iteration. This proves property (iii).

Finally, we verify property (iv). Initially $\mathcal{G}_{R - (S - T)}$ consists only of the roots $U_j$ for $j \in I(S \triangle T)$. In an iteration of the while loop, the chosen vertex $x$ in step 3b lies in $U_i$ for some $i \in I(R)$, while all of its neighbors in the IT $Q_x$ lie in blocks $U_j$ with $j \notin I(R)$. Thus, $\mathcal{G}_{R' - (S' - T')}$ is obtained from $\mathcal{G}_{R - (S - T)}$ by connecting the vertices $U_j$, for $j \in I(Q_x \cap N(x))$, as leaves to the vertex $U_i$.
\end{proof}

Claim \ref{claim-1} thus establishes the validity of Algorithm \ref{alg:growing-tuple}. The goal for rest of this subsection is to prove the following essential lemma.

\begin{lem} \label{lem:induced-matching}
Suppose that $(G, \cU)$ satisfies conditions (a), (b), (c) of Theorem \ref{thm:equivalent-general}. Consider any two IT's $S', T'$ that lie in fixed distinct connected components of $\recongraph(G,\cU)$ and with $|S' \cap T'|$ maximum among such IT's. If we apply Algorithm \ref{alg:growing-tuple} to input $S', T'$, then any output $(R,S,T)$ is an \matchingconfig.
\end{lem}

We prove Lemma \ref{lem:induced-matching} using the following sequence of claims, explicitly noting where we use conditions (a), (b), (c) of Theorem \ref{thm:equivalent-general}. Let $(R,S,T)$ denote an output of Algorithm \ref{alg:growing-tuple}.

\begin{claim} \label{claim-2}
The induced subgraph $G[S \triangle T]$ has no isolated vertices.
\end{claim}

\begin{proof}
If $G[S \triangle T]$ had an isolated vertex $x \in U_j$ (where $j \in I(S \triangle T)$), then we could have started with $S_0 \coloneqq S \cup \{x\} - \{s\}$ and $T_0 \coloneqq T \cup \{x\} - \{t\}$ in place of $S$ and $T$ at step 1, where $s \in U_j \cap S$ and $t \in U_j \cap T$. These satisfy $S_0 \in C_0$, $T_0 \in C_1$, and $|I(S_0 \cap T_0)| > |I(S \cap T)|$, which contradicts the original choice of $S$ and $T$. 
\end{proof}

\begin{claim} \label{claim-3}
Every vertex in $\bigcup_{i \in I(R)} U_i$ lies in a component of $G$ that intersects $R$.
\end{claim}

\begin{proof}
If $x \in \bigcup_{i \in I(R)} U_i$ does not lie in a connected component of $G$ intersecting $R$, then in particular $x$ is not totally dominated by $R$. So we could have continued with another iteration of the while loop in Algorithm \ref{alg:growing-tuple}, which contradicts that $(R,S,T)$ is the output.
\end{proof}

\begin{claim} \label{claim-4}
Every vertex in $(S \cap T) - R$ lies in a component of $G$ that does not intersect $R$.
\end{claim}

\begin{proof}
Suppose that vertex $v \in (S \cap T) - R$ lies in a connected component of $G$ that intersects $R$. Since this component is a complete bipartite graph (by condition (b) of Theorem \ref{thm:equivalent-general}), $v$ is adjacent to a vertex in one of $S - T$, $T - S$, $R \cap (S \cap T)$, or $R - (S \cup T)$. We cannot have any of the first three cases because $S, T$ are IT's containing $v$, so $v$ is adjacent to a vertex $x \in R - (S \cup T)$. But then $x$ was chosen at some iteration of step 3a of the procedure, and $v$ lies in $Q_x$ from step 3b. Hence $v$ lies in the output $R$, a contradiction.
\end{proof}

\begin{claim} \label{claim-5}
The instance $(G[\bigcup_{i \in I(R)} U_i], \{U_i : i \in I(R)\})$ has a disconnected reconfiguration graph.
\end{claim}

\begin{proof}
Suppose for contradiction that the reconfiguration graph is connected. Then there is a reconfiguration sequence of IT's from $R \cap S$ to $R \cap T$, say $R \cap S = S_0', S_1', \ldots, S_{N-1}', S_N' = R \cap T$. Since every vertex in $S'' \coloneqq (S \cap T) - R$ is independent of $R$ by Claim \ref{claim-4}, this sequence can be lifted to a reconfiguration sequence of IT's of $(G, \cU)$ from $S$ to $T$, namely $S = S_0' \cup S'', S_1' \cup S'', \ldots, S_{N-1}' \cup S'', S_N' \cup S'' = T$. This contradicts that $S$ and $T$ lie in distinct connected components of $\recongraph(G, \cU)$.
\end{proof}

\begin{claim} \label{claim-6}
We have $I(R) = [m]$.
\end{claim}

\begin{proof}
Since $(G, \cU)$ is minimally RGD (by condition (a) of Theorem \ref{thm:equivalent-general}), this follows directly from Claim \ref{claim-5}.
\end{proof}

\begin{claim} \label{claim-7}
The induced subgraph $G[R]$ is a matching.
\end{claim}

\begin{proof}
By Claims \ref{claim-3} and \ref{claim-6}, every connected component of $G$ contains a vertex from $R$. By construction, every connected component of $G[R]$ either lies in $G[S \triangle T]$, or is a non-trivial star that lies in $G[R - (S \triangle T)]$ with center in $R - (S \cup T)$ and at least one vertex in $S \cap T$. Since every component of $G[S \triangle T]$ has at least two vertices (by Claim \ref{claim-2}), $G[S \triangle T]$ has at most $|S \triangle T|/2 = |S - T|$ connected components. Thus,
\begin{equation*}
	\# \text{ components of } G \le \# \text{ components of } G[R] \le |S - T| + |S \cap T| = |S| = |\cU|.
\end{equation*}
By condition (c) of Theorem \ref{thm:equivalent-general}, the number of components of $G$ is equal to $|\cU|$. Thus, there must equality throughout the above inequality. In particular, $G[S \triangle T]$ has exactly $|S \triangle T|/2$ components each with at least two vertices, and so $G[S \triangle T]$ is a matching on $S \triangle T$. In addition, every component in $G[R - (S \triangle T)]$ consists of exactly one vertex from $S \cap T$ connected to a vertex in $R - (S \cup T)$, also yielding that $G[R - (S \triangle T)]$ is a matching. Therefore, $G[R]$ is a matching.
\end{proof}

Claims \ref{claim-6} and \ref{claim-7} finish the proof of Lemma \ref{lem:induced-matching}, so that any output $(R,S,T)$ of Algorithm \ref{alg:growing-tuple} is an \matchingconfig.

\subsection{Properties of induced matching configurations}

In this subsection, we record some additional properties of induced matching configurations (\matchingconfig's) $(R, S, T)$ on $(G, \cU)$ that will be useful for the proof of Theorem \ref{thm:equivalent-general} in the following subsection. Again we assume that $(G, \cU)$ satisfies conditions (a), (b), (c) of Theorem \ref{thm:equivalent-general}.

\begin{claim} \label{claim-8}
The block graph $\mathcal{G}_{S \triangle T}$ is a disjoint union of cycles.
\end{claim}

\begin{proof}
Since $(R,S,T)$ is an \matchingconfig, the induced subgraph $G[S \triangle T]$ is a matching. Thus, the block graph $\mathcal{G}_{S \triangle T}$ is $2$-regular, and the claim follows.
\end{proof}

\begin{claim} \label{claim-9}
Every vertex of $G$ is adjacent to exactly one vertex in $R$.
\end{claim}

\begin{proof}
By Claims \ref{claim-2} and \ref{claim-5}, every vertex $v \in V(G)$ lies in a component that intersects $R$. Since $G[R]$ is a matching and every component of $G$ is a complete bipartite graph (by condition (b) of Theorem \ref{thm:equivalent-general}), it follows that $v$ is adjacent to exactly one vertex in $R$. 
\end{proof}

For the last two claims, and part of the proof in the next subsection, we use the following terminology. Let $(R,S,T)$ be an \matchingconfig{} on the vertex-partitioned graph $(G,\cU)$. Recall that the block graph $\mathcal{G}_{R - (S - T)}$ is a forest whose components are rooted at the blocks $U_r$ for $r \in I(S \triangle T)$. Then every other block $U_i$ is a descendant of one of these roots. For a block $U_j$ and vertex $w \in U_i \cap (R - (S \cap T))$, the \textit{$w$-child} of $U_j$ in $\mathcal{G}'$ is the block $U_k$ containing the unique neighbor of $w$ in $S \cap T$. A block $U_k$ is a \textit{$w$-descendant} of block $U_j$ if there is a sequence of blocks $U_j = V_0, V_1, \ldots, V_{N-1}, V_N = U_k$ together with a sequence of vertices $w = w_0, v_1, w_1, v_2, \ldots, v_{N-1}, w_{N-1}, v_N$ such that $w_i \in V_i \cap (R - (S \cup T))$, $v_i \in V_i \cap (S \cap T)$, and $v_i$ is the unique neighbor of $w_{i-1}$ in $R$, for all $i$. That is, each $V_i$ is the $w_{i-1}$-child of $V_{i-1}$. The sequence of blocks $V_0, V_1, \ldots, V_{N-1}, V_N$ is the unique path from $U_j$ to $U_k$ in the block graph $\mathcal{G}_{R - (S \triangle T)}$.

\begin{claim} \label{claim-10}
Let $s \in S - T$ and $t \in T - S$ be neighbors in $G[S \triangle T]$, lying in blocks $U_k$ and $U_\ell$ respectively. Then $N_G(s) \subseteq U_\ell$ and $N_G(t) \subseteq U_k$.
\end{claim}

\begin{proof}
We show that $N_G(t) \subseteq U_k$, while a symmetric argument would show that $N_G(s) \subseteq U_\ell$. We show that if $N_G(t) \not\subseteq U_k$, then we may construct another IT $S''$ in the same component of $\recongraph(G,\cU)$ as $S$ and satisfying $|I(S'' \cap T)| > |I(S \cap T)|$, which would contradict the maximality of $|I(S \cap T)|$ in the definition of feasible 3-tuples. 

See Figure \ref{fig:switching-1}(a) for an illustration of the argument. Supposing that $N_G(t) \not\subseteq U_k$, there exists a vertex $x \in N_G(t) - U_k$. We have $x \notin R$ because $G[R]$ is a matching, and $t$ is the unique neighbor of $x$ in $R$ by Claim \ref{claim-9}. Say that $x$ lies in block $U_{j} \in \cU$, and that $U_{j}$ is a descendent of root $U_r$ in the block graph $\mathcal{G}_{R - (S - T)}$. Let $U_r = V_0, V_1, \ldots, V_{N-1}, V_N = U_{j}$ be the unique path from $U_r$ to $U_{j}$ in $\mathcal{G}_{R - (S - T)}$, with associated sequence of vertices $w_0, v_1, w_1, v_2, \ldots, v_{N-1}, w_{N-1}, v_N$ such that $w_i \in V_i \cap (R - (S \cup T))$, $v_i \in V_i \cap (S \cap T)$, and $v_i$ is the unique neighbor of $w_{i-1}$ in $R$, for all $i$. First we construct the sequence of IT's $S, S_0, S_1, \ldots, S_{N-1}, S_N$ where
\begin{equation*}
	S_0 \coloneqq S \cup \{x\} - \{v_N\}, \quad \quad S_i \coloneqq S_{i-1} \cup \{w_{N-i}\} - \{v_{N-i}\} \text{ for all } 1 \le i \le N, 
\end{equation*} 
with $v_0 = s_0$ denoting the vertex in $U_r \cap S$.

Since the block graph $\mathcal{G}_{S \triangle T}$ is a disjoint union of cycles by Claim \ref{claim-8}, the connected component of $\mathcal{G}_{S \triangle T}$ containing $U_r$ is the cycle $U_r = V_0', V_1', \ldots, V_{M-1}', V_M' = U_r$, with associated sequence of vertices $s_0, t_1, s_1, t_2, \ldots, t_{M-1}, s_{M-1}, t_0, s_0$ where $s_i \in V_i' \cap S$ and $t_i \in V_i' \cap T$ for all $i$. Then we may construct the sequence of IT's $S_N = S_0', S_1', \ldots, S_{M-1}', S_M'$ where
\begin{equation*}
	S_i' \coloneqq S_{i - 1}' \cup \{t_i\} - \{s_i\} \text{ for all } 1 \le i \le M-1, \quad \quad S_M' \coloneqq S_{M-1} \cup \{t_0\} - \{w_0\}.
\end{equation*}
We can do this reconfiguration regardless of whether or not the blocks $U_k, U_\ell$ containing $s,t$ are among the $V_i'$ (this is the case in Figure \ref{fig:switching-1}(a)).

Finally, we do a reversal of the first reconfiguration sequence to get the sequence of IT's $S_M' = S_0'', S_1'', \ldots, S_{N-1}'', S_N''$ where
\begin{equation*}
	S_i'' \coloneqq S_{i - 1}'' \cup \{v_i\} - \{w_i\} \text{ for all } 1 \le i \le N-1, \quad \quad S_N'' = S_{N-1} \cup \{v_N\} - \{x\}.
\end{equation*}
Then the IT $S'' \coloneqq S_N''$ still lies in the same component of $\recongraph(G, \cU)$ as $S$, and now it agrees with $T$ on the additional blocks $V_0', V_1', \ldots, V_{M-1}'$. This gives the required construction.
\end{proof}

\begin{figure}[h]
	\begin{center}
		\includegraphics[scale=0.75]{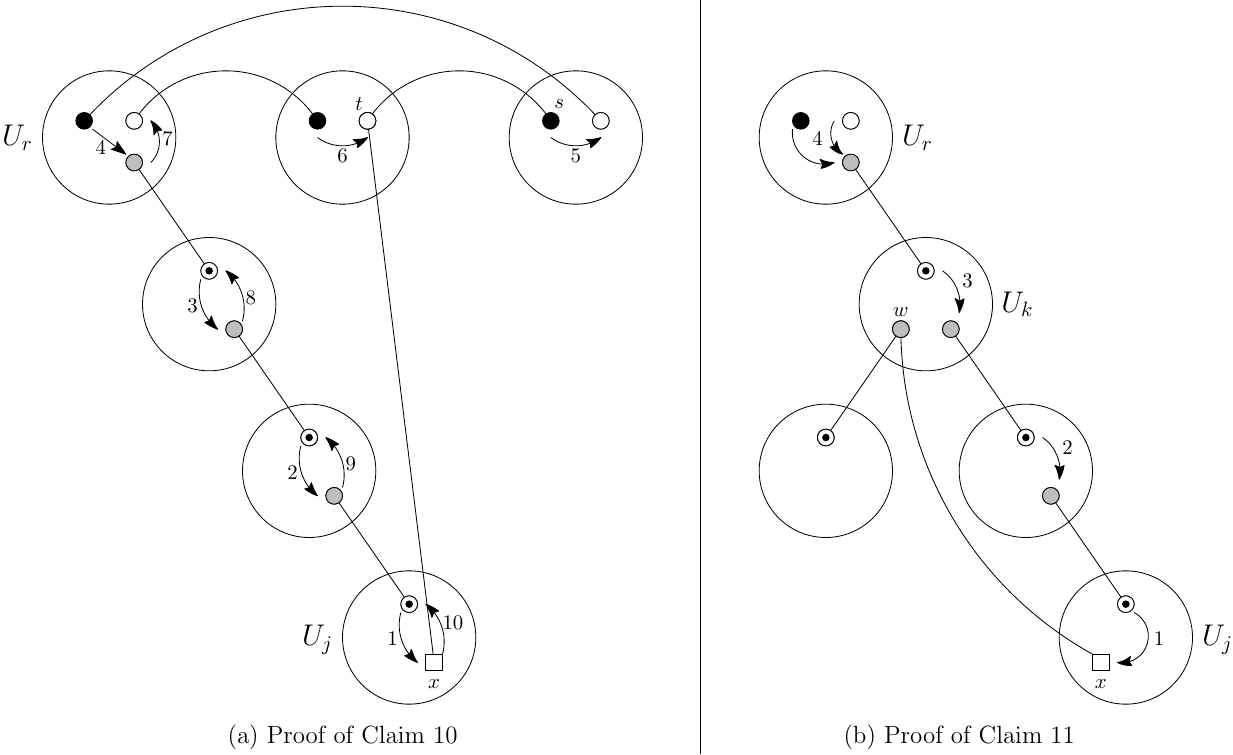}
	\end{center}
	\caption{(a) Reconfiguration in proof of Claim \ref{claim-10}. The arrows with numbers represent the movement of the black vertex within a block at a given step. (b) Reconfigurations in proof of Claim \ref{claim-11}. The arrows with numbers represent the movement of both the black and white vertex within a block at a given step.}
	\label{fig:switching-1}
\end{figure}

\begin{claim} \label{claim-11}
For every vertex $w \in R - (S \cup T)$, lying in block $U_k \in \cU$, every neighbor of $w$ in $G$ lies in a block that is a $w$-descendant of $U_k$ in the block graph $\mathcal{G}_{R - (S - T)}$.
\end{claim}

\begin{proof}
See Figure \ref{fig:switching-1}(b) for an illustration of this proof. Suppose for contradiction that $w$ has a neighbor $x$ in a block $U_j$ that is not a $w$-descendant of $U_k$ in $\mathcal{G}_{R - (S - T)}$. By Claim \ref{claim-9}, $w$ is the unique neighbor of $x$ in $R$. Say that block $U_j$ is a descendant of root $U_{r}$ in $\mathcal{G}_{R - (S - T)}$, where $r \in I(S \triangle T)$. Let $U_{r} = V_0, V_1, \ldots, V_{N-1}, V_N = U_j$ be the unique path from $U_{r}$ to $U_j$ in $\mathcal{G}_{R - (S - T)}$, and let $w_0, v_1, w_1, v_2, \ldots, v_{N-1}, w_{N-1}, v_N$ be the associated sequence of vertices of $R$ where $w_i \in V_i \cap (R - (S \cup T))$, $v_i \in V_i \cap (S \cap T)$, and $v_i$ is the unique neighbor of $w_{i-1}$ in $R$. None of the vertices in the sequence are $w$ (although one of the blocks could contain $w$, which is the case in Figure \ref{fig:switching-1}(b)), so none of these vertices are adjacent to $x$. Then we are able to produce the following reconfiguration sequences of IT's of $(G, \cU)$: $S, S_0, S_1, \ldots, S_{N-1}, S_N$ and $T, T_0, T_1, \ldots, T_{N-1}, T_N$, where
\begin{equation*}
	S_0 \coloneqq S \cup \{x\} - \{v_N\}, \quad S_{i} \coloneqq S_{i-1} \cup \{w_{N-i}\} - \{v_{N-i}\} \text{ for all } 1 \le i \le N,
\end{equation*}
and $v_0 = s$ is the vertex in $U_{\ell} \cap S$ (and analogous definitions for the $T_i$'s). Since $S_N, T_N$ still lie in the same connected component of $\recongraph(G,\cU)$ as $S, T$, and $|I(S_N \cap T_N)| > |I(S \cap T)|$, this contradicts the maximality of $|I(S \cap T)|$ in the definition of feasible tuples.
\end{proof}

\subsection{Finishing the proof of Theorem \ref{thm:equivalent-general}}

In this subsection, we finish the proof of Theorem \ref{thm:equivalent-general}. Let $(G,\cU)$ satisfy conditions (a), (b), (c) of the theorem statement. We prove that conclusion of the theorem holds for the block $U_m \in \cU$, without loss of generality. It is sufficient to find a vertex $u \in V(G)$ with neighborhood $N_G(u) \subseteq U_m$. This is because, by condition (b), every component of $G$ is a complete bipartite graph, so $N_G(u)$ is a part $A$ of some complete bipartite component $K_{A,B}$ of $G$, and we get that $A \subseteq U_m$ as required. 

We start by fixing two distinct connected components $C_0, C_1$ of $\recongraph(G,\cU)$, and then choosing two IT's $S' \in C_0$, $T' \in C_1$ such that $|I(S' \cap T')|$ is maximized. Then we let $(R, S, T)$ be the \matchingconfig{} that is output from Algorithm \ref{alg:growing-tuple} when applied to $S', T'$. First assume that $m \in I(S \triangle T)$, that is, $S$ and $T$ disagree at the block $U_m$. Let $s \in U_m \cap S$ be the representative of $U_m$ in $S$, and let $t \in T - S$ be the unique neighbor of $s$ in $R$ (by Claim \ref{claim-9}). Then by Claim \ref{claim-10}, we already have $N_G(t) \subseteq U_m$, so $u \coloneqq t$ is a desired vertex.

Now assume instead that $m \in I(S \cap T)$, that is, $S$ and $T$ agree at the block $U_m$. We prove that a desired vertex exists by induction on the degree of the block $U_m$ in the block forest $\mathcal{G}_{R - (S - T)}$. To start, suppose that the degree of $U_m$ is one, that is, it is a leaf. Say that $U_m$ is the $w$-descendant of block $U_j$ in $\mathcal{G}_{R - (S - T)}$, for some $1 \le j \le m-1$ and $w \in (R - (S \cup T)) \cap U_j$. By Claim \ref{claim-11}, every neighbor of $w$ in $G$ lies in a block that is a $w$-descendant of $U_j$. But $U_m$ is the only $w$-descendant of $U_j$ (as it is a leaf), so $N_G(w) \subseteq U_m$ and $u \coloneqq w$ is a desired vertex.

Assuming that a desired vertex exists when the degree is $d-1$, for some $d \ge 2$, suppose now that $U_m$ has degree $d$ in $\mathcal{G}_{R - (S - T)}$. Let $v \in U_m \cap (S \cap T)$, and let $w \in U_j$ be the unique neighbor of $v$ in $R$, so that $U_m$ is the $w$-child of $U_j$. If $N_G(w) \subseteq U_m$, then $u \coloneqq w$ is our desired vertex. Otherwise, there exists some vertex $x \in N_G(w) - U_m$, lying in block say $U_k$. By Claim \ref{claim-10}, $U_k$ is a $w$-descendant of $U_j$, and hence a descendant of $U_m$. Let $U_j = V_0, V_1, V_2 \ldots, V_{N-1}, V_N = U_k$ be the unique path from $U_j$ to $U_k$ in $\mathcal{G}_{R - (S - T)}$. Let $w_0, v_1, w_1, v_2, \ldots, v_{N-1}, w_{N-1}, v_N$ be the associated sequence of vertices of $R$ where $w_i \in V_i \cap (R - (S \cup T))$, $v_i \in V_i \cap (S \cap T)$, and $v_i$ is the unique neighbor of $w_{i-1}$ in $R$. Note that $V_1 = U_m$, $w_0 = w$, and $v_1 = v$. We construct the reconfiguration sequences of IT's $S = S_0, S_1, \ldots, S_{N-1}, S_{N}$ and $T=T_0, T_1, \ldots, T_{N-1}, T_{N}$ where
\begin{align*}
	S_1 \coloneqq S \cup \{x\} - \{v_N\}, \quad \quad S_i \coloneqq S_{i-1} \cup \{w_{N-i+1}\} - \{v_{N-i+1}\} \text{ for all } 2 \le i \le N,
\end{align*}
and analogous definitions for the $T_i$'s. Also defining $R_N \coloneqq R \cup \{x\} - \{v\}$, we claim that $(R_{N}, S_{N}, T_{N})$ is another \matchingconfig. 

Indeed, we still have $I(R_N) = [m]$, and $G[R_N]$ is still a matching after replacing the edge $\{w,v\}$ with the edge $\{w,x\}$. Properties (i), (ii), and (iii) of feasible tuples are easily seen to still be satisfied. Finally, to verify property (iv), notice that the block graph $\mathcal{G}_{R_N - (S_{N} - T_{N})}$ is obtained from the block graph $\mathcal{G}_{R - (S - T)}$ by adding the edge $\{U_j, U_k\}$ while deleting the edge $\{U_j, U_m\}$. The edge $\{U_j, U_m\}$ was the first edge in the unique path from $U_j$ to $U_k$ in the forest $\mathcal{G}_{R - (S - T)}$, so hence $\mathcal{G}_{R_N - (S_{N} - T_{N})}$ is still a forest. 

Now $U_m$ has degree one less in $\mathcal{G}_{R_N - (S_{N} - T_{N})}$ than it did in $\mathcal{G}_{R - (S - T)}$, since it lost the incident edge $\{U_j, U_m\}$ while all other incident edges remain the same. Therefore, a desired vertex $u \in V(G)$, with $N_G(u) \subseteq U_m$, exists by induction. Our proof of Theorem \ref{thm:equivalent-general} is finished.

\section{Concluding remarks} \label{sec:discussion}

In this paper, we solved Problem \ref{prob:main-problem} of Buys, Kang, and Ozeki \cite{buys2025reconfiguration}, by giving an exact characterization of instances $(G,\cU)$ that are extremal for their Theorem \ref{thm:BKO}. Our solution was analogous to former combinatorial work of Haxell and the author \cite{haxell2024constructing} on graphs with no independent transversals, although we required a modified technical setup.

In \cite{haxell2024constructing}, the analogous construction lemma \cite[Lemma 3]{haxell2024constructing} was applied to many different settings beyond disjoint unions of complete bipartite graphs. For example, working in a list coloring setting based on color degree, the authors systematically generalized counterexamples due to Bohman and Holzman \cite{bohman2002list} to a former conjecture of Reed \cite{reed1999list}. See \cite{cambie2024precise, haxell2024degree, wdowinski2026bounded} for other interesting applications of their lemma. The construction lemma in this paper (Lemma \ref{lem:construction-intro}) applies equally well to all those settings, just now with a different starting point and the results being disconnected reconfiguration graphs (as opposed to empty ones). On the other hand, in settings such as list coloring which involve rather large numbers of blocks, Lemma \ref{lem:construction-intro} does not appear to yield much of an improvement over the previous constructions or even more elementary constructions. It would be interesting to find if one can do better than Lemma \ref{lem:construction-intro} in these kinds of settings.

Analogous to \cite[Lemma 24]{haxell2024constructing}, Lemma \ref{lem:construction-intro} can be generalized to the setting of simplicial complexes. In this generalized setting, independent transversals in graphs are replaced by colorful simplices in simplicial complexes. The reconfiguration graph on colorful simplices is defined analogously to independent transversals (see \cite{wdowinski2025hall}), and the union of disjoint graphs used in Lemma \ref{lem:construction-intro} is replaced by the \textit{join} of two simplicial complexes (see \cite{haxell2024constructing}). Such a generalized lemma could be useful in other situations where one seeks disconnected reconfiguration graphs not necessarily phrased in terms of independent transversals. The main challenge would be to find when the construction lemma precisely characterizes the extremal examples, which was the situation in this paper but is definitely not always the case. From a simplicial complex perspective, the main result presented in this paper is a characterization of vertex partitions on certain triangulated spheres whose reconfiguration graphs are disconnected.

Finally, in view of the topological connections made in \cite{wdowinski2025hall}, it is natural to give a common generalization of \cite[Lemma 3]{haxell2024constructing} and our Lemma \ref{lem:construction-intro} in terms of higher dimensional topological connectedness of the \textit{colorful complex} $\colorfulcomplex(\mathcal{I}(G),\cU)$ which was described in \cite{wdowinski2025hall}. Such a generalized construction lemma does indeed hold, although we do not make it explicit here. On the other hand, it is not immediate to generalize \cite[Lemma 28]{haxell2024constructing} and Theorem \ref{thm:equivalent-general} to this higher dimensional setting. This is due to not yet having the appropriate combinatorial tools to prove them. Such generalizations would yield an analogous characterization of vertex-partitioned disjoint unions of complete bipartite graphs whose colorful complex $\colorfulcomplex(\mathcal{I}(G),\cU)$ is not homologically $k$-connected, where $k \ge -1$ is fixed. The following conjecture is what we would like to show (see \cite{wdowinski2025hall} for context and notation, reduced homology $\widetilde{H}_j(\cdot)$ being over any fixed ring).

\begin{conj} \label{conj:general-topology}
Let $G$ be a graph, let $\cU = \{U_1, \ldots, U_m\}$ be a partition of $V(G)$, and let $k \ge -1$. Suppose that
\begin{itemize}
	\item[(a)] $\widetilde{H}_k(\colorfulcomplex(\mathcal{I}(G),\cU)) \neq 0$, but $\widetilde{H}_j(\colorfulcomplex(\mathcal{I}(G - U_i),\cU - \{U_i\})) = 0$ for all $U_i \in \cU$ and $-1 \le j \le k$;
	\item[(b)] $G$ is a disjoint union of complete bipartite graphs; and
	\item[(c)] $G$ has exactly $|\cU|+k$ connected components.
\end{itemize}
Then for every $U_i \in \cU$, there exists a complete bipartite component $K_{A,B}$ of $G$ such that $A \subseteq U_i$.
\end{conj}

\section*{Acknowledgments}
This research was supported by the Austrian Science Fund (FWF) [10.55776/F1002]. For open access purposes, the author has applied a CC BY public copyright license to any author accepted manuscript version arising from this submission.

\printbibliography

\newpage

\appendix

\section{Proof of Lemmas \ref{lem:refined-construction} and \ref{lem:minimal-non-reconfigurable}} \label{sec:appendix}

In this appendix, we provide proofs of Lemmas \ref{lem:refined-construction} and \ref{lem:minimal-non-reconfigurable}. 

\begin{proof}[Proof of Lemma \ref{lem:refined-construction}]
We prove the following four separate claims, assuming throughout that $(H, \cV)$ is minimally NIT:
\begin{itemize}
	\item[(i)] If $\recongraph(G, \cU)$ is disconnected, then $\recongraph(G \cup H, \cW)$ is disconnected.
	\item[(ii)] If $\recongraph(G \cup H, \cW)$ is disconnected, then $\recongraph(G, \cU)$ is disconnected.
	\item[(iii)] If $\recongraph(G - U, \cU - \{U\})$ is connected for all $U \in \cU$, then $\recongraph((G \cup H) - W, \cW - \{W\})$ is connected for all $W \in \cW$.
	\item[(iv)] If $\recongraph((G \cup H) - W, \cW - \{W\})$ is connected for all $W \in \cW$, then $\recongraph(G - U, \cU - \{U\})$ is connected for all $U \in \cU$.
\end{itemize}
Claim (i) follows from Lemma \ref{lem:construction-intro}, which was already proven at the beginning of Section \ref{sec:construction-lemma}.

For claim (ii), we prove the contrapositive: Assuming that $\recongraph(G, \cU)$ is connected, we wish to show that $\recongraph(G \cup H, \cW)$ is connected. Consider any two IT's $S, T$ of $(G \cup H, \cW)$. Since $(H, \cW)$ is minimally NIT, the sets $S \cap V(G)$ and $T \cap V(G)$ each contain an IT of $(G, \cU)$ (they could contain multiple vertices from $U_m$), say $S_0$ and $T_0$ respectively. Since $\recongraph(G, \cU)$ is connected, there is a reconfiguration sequence of IT's of $(G, \cU)$ from $S_0$ to $T_0$, say $S_0, S_1, S_2, \ldots, S_{N-1}, S_N = T_0$. Since $(H, \cV)$ is minimally NIT, each of these IT's of $(G, \cU)$ can be extended to IT's of $(G \cup H, \cW)$, say $S_0 \cup S_0', S_1 \cup S_1', \ldots, S_{N-1} \cup S_{N-1}', S_N \cup S_N'$. Assume that $S_i' = S_j'$ whenever $S_i \cap U_m$ and $S_j \cap U_m$ lie in the same block $V_j' \in \cW$. Now, for each $0 \le i \le N-1$, we claim that we can find a reconfiguration sequence of IT's of $(G \cup H, \cW)$ from $S_i \cup S_i'$ to $S_{i+1} \cup S_{i+1}'$. If $S_i \cap U_m$ and $S_{i+1} \cap U_m$ lie in the same block $V_j' \in \cW$, then we may reconfigure $S_i \cup S_i'$ to $S_{i+1} \cup S_{i+1}' = S_{i+1} \cup S_i'$ directly. Otherwise, if $S_i \cap U_m$ lies in $V_{j_0}' \in \cV$ and $S_{i+1} \cap U_m$ lies in $V_{j_1} \in \cV$ with $j_0 \neq j_1$, then we first reconfigure $S_i \cup S_i'$ to $(S_i \cup S_{i+1}) \cup (S_i' - V_{j_1})$, and then using that $\recongraph(H - V_{j_0} - V_{j_1}, \cV - \{V_{j_0}, V_{j_1}\})$ is connected, we find a reconfiguration sequence of IT's of $(G \cup H, \cW)$ from $(S_i \cup S_{i+1}) \cup (S_i' - V_{j_1})$ to $(S_i \cup S_{i+1}) \cup (S_{i+1}' - V_{j_0})$, concluding by reconfiguring the latter to $S_{i+1} \cup S_{i+1}'$. Finally, a similar argument can be used to find a reconfiguration sequence of IT's of $(G \cup H, \cW)$ from $S$ to $S_0 \cup S_0'$, and from $S_N \cup S_N'$ to $T$. 
	
For claim (iii), assuming that $\recongraph(G - U, \cU - \{U\})$ is connected for all $U \in \cU$, we wish to show that $\recongraph((G \cup H) - W, \cW - \{W\})$ is connected for all $W \in \cW$. First note that in the case $W = U_j$ for some $1 \le j \le m - 1$, this claim follows from claim (ii), where $(G, \cU)$ is replaced by $(G - U_j, \cU - \{U_j\})$. So instead assume that $W = V_j'$ for some $1 \le j \le n$. Consider any two IT's $S, T$ of $((G \cup H) - V_j', \cW - \{V_j'\})$. Using that $(H, \cV)$ is minimally NIT, fix any IT $S'$ of $(H - V_j, \cV - \{V_j\})$. We find a reconfiguration sequence of IT's of $\recongraph((G \cup H) - V_j', \cW - \{V_j'\})$ first from $S$ to $(S - U_m) \cup S'$, then from $(S - U_m) \cup S'$ to $(T - U_m) \cup S'$, and finally from $(T - U_m) \cup S'$ to $T$. For the second part, we simply use the assumption $\recongraph(G - U_m, \cU - \{U_m\})$ is connected. For the first part, suppose that $S \cap U_m$ consists of vertices from the blocks $V_{j_1}', \ldots, V_{j_k}'$ (possibly $S \cap U_m = \emptyset$). Using that $\recongraph(H - V_j - V_{j_1}' - \cdots - V_{j_k}', \cV - \{V_j, V_{j_1}', \ldots V_{j_k}'\})$ is connected, we can find a reconfiguration sequence of IT's of $((G \cup H) - V_j', \cW - \{V_j'\})$ from $S$ to $(S \cap V(G)) \cup (S' - V_{j_1} - \cdots - V_{j_k})$, and then to $S \cup S'$. The third part is analogous to the first part.
	
For claim (iv), assuming that $\recongraph((G \cup H) - W, \cW - \{W\})$ is connected for all $W \in \cW$, we wish to show that $\recongraph(G - U, \cU - \{U\})$ is connected for all $U \in \cU$. First note that if $U = U_j$ for some $1 \le j \le m - 1$, then the claim follows from claim (i), where $(G, \cU)$ is replaced by $(G - U_j, \cU - \{U_j\})$. So instead assume that $U = U_m$. Consider any two IT's $S, T$ of $(G, \cU)$. These can be extended to IT's of $((G \cup H) - V_n', \cW - \{V_n'\})$, say $S_0$ and $T_0$ respectively. Since $\recongraph((G \cup H) - V_n', \cW - \{V_n'\})$ is connected, we can find a reconfiguration sequence of IT's of $((G \cup H) - V_n', \cW - \{V_n'\})$ from $S_0$ to $T_0$, say $S_0, S_1, S_2, \ldots, S_{N-1}, S_N = T_0$. Then the sequence $S = (S_0 \cap V(G)) - U_m, (S_1 \cap V(G)) - U_m, \ldots, (S_{N-1} \cap V(G)) - U_m, (S_N \cap V(G)) - U_m = T$ forms a reconfiguration sequence of IT's of $(G - U_m, \cU - \{U_m\})$ from $S$ to $T$, after collapsing consecutive copies of IT's.
\end{proof}

\begin{proof}[Proof of Lemma \ref{lem:minimal-non-reconfigurable}]
Let $(G, \cU)$ be an elementary non-reconfigurable instance, with $G$ a disjoint union of $m$ complete bipartite graphs and $\cU = \{U_1, \ldots, U_m\}$ a partition of $V(G)$. Applying the definition, say that block $U_i \in \cU$ is associated with component $K_{A_i,B_i}$ of $G$ with $A_i \subseteq U_i$. First notice that $\recongraph(G, \cU)$ is nonempty because we can form an IT $S$ of $(G,\cU)$ by choosing one verex from each of $A_1, \ldots, A_n$. Now we show that $\recongraph(G, \cU)$ is disconnected. Observe that the connected component of $\recongraph(G,\cU)$ containing $S$ consists only of IT's of $(G,\cU)$ formed by choosing one vertex from each of $A_1, \ldots, A_n$. Thus, it suffices to find another IT $T$ of $(G, \cU)$ which contains a vertex from $U_i - A_i$ for some $i \in [m]$. We use the following algorithm.

\begin{algorithm}\caption{Constructing another IT of a minimally non-reconfigurable instance} \label{alg:find-another}
\begin{enumerate}
	\item Initiate $G' \leftarrow G$, $U_i' \leftarrow U_i$ for all $i \in [m]$, $I \leftarrow [m]$, and $T \leftarrow \emptyset$.
	\item Initiate $J \leftarrow \{i \in I : U_i' = A_i\}$. 
	\item While $J \neq \emptyset$:
	\begin{itemize}
		\item[a.] Choose any $u_i \in U_{i}'$ for all $i \in J$, and update $T \leftarrow T \cup \{u_{i} : i \in J\}$.
		\item[b.] Update $G' \leftarrow G' - \bigcup_{i \in J} K_{A_i, B_i}$, $U_i' \leftarrow U_i' \cap V(G')$ for all $i \in I - J$, and $I \leftarrow I - J$.
		\item[d.] Update $J \leftarrow \{i \in I : U_i' = A_i\}$.
	\end{itemize}
	\item Choose any $u_i \in U_i' - A_i$ for all $i \in I$, 
	\item Update $T \leftarrow T \cup \{u_i : i \in I\}$, and return $T$.
\end{enumerate}
\end{algorithm}
	
Notice that Algorithm \ref{alg:find-another} terminates because the cardinality of $I$ decreases after each loop, so eventually $J = \emptyset$. Notice also that all of the $U_i'$ remain nonempty after each iteration, since we start with an elementary non-reconfigurable instance. We claim that the output $T$ of the algorithm is an IT of $(G, \cU)$. Assume that at the start of an iteration, $T$ is an IT of $(G - G', \{U_i' : i \in [m] - I\})$, where $G' = \bigcup_{i \in [m] - I} K_{A_i, B_i}$. If $J = \emptyset$, then we update $T$ by taking a union with an IT of $(G', \{U_i' : i \in I\})$, thus making $T$ an IT of $(G, \cU)$. Otherwise, if $J \neq \emptyset$, then we update $T$ by taking a union with a transversal of $\{U_i' : i \in J\}$, which is necessarily independent. Thus, with the updated $G'$, $U_i'$, and $I$, this makes $T$ an IT of $(G - G', \{U_i' : i \in [m] - I\})$, as required.
	
Now we claim that $T$ contains a vertex from $U_i - A_i$ for some $i \in [m]$. This comes from noticing that it is impossible that $J = I$ in any given iteration (since all of the $B_i$'s are nonempty), so there will be an iteration where $I \neq \emptyset$ and $J = \emptyset$. After this last iteration, $T$ will be appended with a vertex $u_i \in U_i - A_i$ for all $i \in I$. This finishes the proof.
\end{proof}

\end{document}